\documentclass[11pt,a4paper,oneside]{amsart}
\usepackage[utf8]{inputenc}

\usepackage{import}
\input{preamble.sty}

\title{Product separability in central extensions}
\author{Lawk Mineh \orcidlink{0000-0001-7144-0124}}

\address{Mathematisches Institut, Universit\"at Bonn, 53115, Germany}
\email{lawk@math.uni-bonn.de}

\begin{document}

\begin{abstract}
    We show that a central extension of locally quasiconvex subgroup separable hyperbolic group is product separable, so long as it is subgroup separable.
    We also establish that a central extension of a double coset separable group by a finitely generated group is double coset separable if and only if it is subgroup separable, and that double coset separability is stable under taking direct products with finitely generated nilpotent groups.
\end{abstract}

\maketitle

\section{Introduction}

Any group \(G\) can be equipped with the profinite topology by declaring (left) cosets of finite index subgroups to be a basis of open subsets.
A subset \(U \subseteq G\) is called \emph{separable} if it is closed with respect to the profinite topology on \(G\).
Determining separability of subgroups and subsets has become an important theme in geometric group theory.
In this paper we investigate separability properties in certain types of group extensions.

We say that \(G\) is \emph{residually finite} if the trivial subgroup of \(G\) is separable, \emph{subgroup separable} (or \emph{LERF} -- locally extended residually finite) if every finitely generated subgroup of \(G\) is separable, \emph{double coset separable} if every double coset of finitely generated subgroups of \(G\) is separable and, for a fixed natural number \(n\), \emph{\(n\)-coset separable} if the setwise product of any \(n\) finitely generated subgroups is separable.
In the limiting case, we say \(G\) is \emph{product separable} if it is \(n\)-coset separable for all natural numbers \(n\).
This property is sometimes referred to as the \emph{Ribes--Zalesskii property} in the literature.

Until recently, very few groups were known to be product separable.
Finitely generated abelian groups are trivially product separable, as they are subgroup separable and every product of subgroups is itself a subgroup, and the property passes to subgroups and finite index overgroups.
Free groups were the first non-trivial examples of groups shown to possess this property: this was proven by Ribes and Zalesskii using the theory of profinite trees \cite{RibesZal}.
With language theoretic techniques, Coulbois showed that the property is preserved under free products and some amalgams with free groups \cite{Coulbois_Free_prod,Coulbois-thesis}.
The author and Minasyan showed using hyperbolic geometry that many nonpositively curved groups are product separable, including all Kleinian groups and subgroup separable fundamental groups of graphs of free groups with cyclic edge groups \cite{MinMin}.

Product separability is a very strong residual property, and has a variety of surprising applications and equivalences.
The property was first considered in its connection with finite semigroup theory: Pin and Reutenauer showed that the Rhodes' conjecture, a statement about the computability of an important subsemigroup of a finite semigroup, reduces to product separability in free groups \cite{Pin_Reutenauer}.
Shepherd recently showed the separability of products of \emph{convex-cocompact} subgroups in specially cubulated groups, and applied this to their actions on their contact graphs \cite{Shepherd}.
It has also been a key ingredient in establishing extensions of partial automorphisms of various structures \cite{Herwig-Lascar,Abdendi-Wise}.
Moreover, for a countable group, product separability is equivalent to the finite approximability of isometric actions of the group on metric spaces \cite{Rosendal}.
We also note that while the weaker notion of double coset separability is more ubiquitous, there are classes of groups whose double coset separability is only known through proofs of their product separability.

A finitely generated group \(G\) is \emph{hyperbolic} if it admits a proper and cocompact action by isometries on a hyperbolic metric space.
Archetypal examples of hyperbolic groups include (small cancellation quotients of) free groups and fundamental groups of compact hyperbolic manifolds.

In this paper we deal with finitely generated subgroups of extensions of hyperbolic groups.
Among the class of hyperbolic groups themselves, arbitrary finitely generated subgroups may be quite wild \cite{Rips}.
On the other hand, the \emph{quasiconvex} subgroups of hyperbolic groups -- subgroups that act properly and cocompactly on a quasiconvex subspace -- are relatively well-behaved.
Indeed, they are necessarily hyperbolic themselves.
A hyperbolic group in which all finitely generated subgroups are quasiconvex is called \emph{locally quasiconvex}; examples include the fundamental groups of free groups, surface groups \cite{Pittet}, and some small cancellation groups \cite{McC-WiseII,McC-Wise}.

Pushing the methods of \cite{RibesZal} further, You extended the product separability of free groups to groups of the form \(F \times \ZZ\) where \(F\) is a free group \cite{You}.
We view this as the simplest kind of central extension of a free group which is not immediately seen to be product separable.
On the other hand, Minasyan showed that hyperbolic groups that are locally quasiconvex and subgroup separable are in fact product separable \cite{MinGFERF}.
Our main result builds on both of these statements.

\begin{theorem}
\label{thm:main_intro}
    Let \(G\) be a finitely generated central extension of a subgroup separable locally quasiconvex hyperbolic group.
    If \(G\) is subgroup separable, then it is product separable.
\end{theorem}

In fact, we establish an a priori more general statement: see Theorem~\ref{thm:main}.
Some assumption on the profinite topology of the extension \(G\) above is to be expected, as in general central extensions do not preserve residual properties \cite{Deligne}.
Note that the centrality of the extensions is essential in the above: there are split extensions \(G = A \rtimes Q\), where \(A\) is finitely generated free abelian and \(Q\) is finitely generated virtually free, such that \(G\) is subgroup separable but not double coset separable \cite{Min_dblcoset}.

As a consequence of the above we obtain

\begin{corollary}
\label{cor:ext_limitgp}
    A finitely generated central extension of a hyperbolic limit group is product separable.
    In particular, the fundamental group of any Seifert-fibred 3-manifold with hyperbolic base orbifold is product separable.
\end{corollary}

Product separability of limit groups (which are, in general, toral relatively hyperbolic) was proven by the author and Minasyan \cite{MinMin}.
In general, the above is false for all limit (or even surface) groups: the integral Heisenberg group \(\operatorname{Heis}_3\ZZ\) is a cyclic central extension of \(\ZZ^2\), but is not even triple coset separable \cite{L-W}.
This further reinforces the apparent strong link between negative curvature and separability properties.

We also establish some results concerning double coset separability in extensions.

\begin{theorem}
\label{thm:central_dblcost}
    Let \(G\) be a central extension of a double coset separable group by a finitely generated group.
    Then \(G\) is double coset separable if and only if it is subgroup separable.
\end{theorem}

The above may be compared with the main result of \cite{RZ_one-relator}.
Again, the example of a split extension noted above shows that centrality of the extension is a necessary assumption in the theorem.
The simplest form of a central extension is a direct product with an abelian group: these are the split central extensions.
Generalising such an extension in a different manner, we may consider direct products with non-abelian groups.
The same methods as above prove:

\begin{theorem}
\label{thm:prod_with_nilp}
    Let \(G\) be a double coset separable group and \(N\) a finitely generated nilpotent group.
    Then the direct product \(G \times N\) is double coset separable.
\end{theorem}

One might reasonably expect that \(G \times P\) is double coset separable for \(P\) polycyclic, though our method does not seem to go through in this case.
Indeed, centrality is still a key factor in the argument and polycyclic groups may be centreless, whereas infinite nilpotent groups have infinite centres. 
Such a statement would probably constitute the best closure property for double coset separability one can hope for as far as taking a product is concerned; Lennox and Wilson proved that the all polycyclic groups are double coset separable, while in general even finitely generated metabelian groups fail to be \cite{L-W}.

\medskip

\textbf{Acknowledgements.} The author would like to thank Ashot Minasyan for many helpful discussions and Corentin Bodart for valuable comments on an earlier version of this paper. This work received funding from the European Union (ERC, SATURN, 101076148) and the Deutsche Forschungsgemeinschaft (EXC-2047/1 -- 390685813).

\section{Preliminaries}

In this section we collect some essential preliminary material.
Let us recall the following standard definitions.

\begin{definition}[Hyperbolic metric space]
    Let \(X\) be a geodesic metric space.
    Given a constant \(\delta \geq 0 \), a geodesic triangle in \(X\) is \emph{\(\delta\)-slim} if each of its sides is contained in a \(\delta\)-neighbourhood of the other two.
    We say \(X\) is a \emph{hyperbolic metric space} if there is \(\delta \geq 0\) such that every geodesic triangle in \(X\) is \(\delta\)-slim.
\end{definition}

Note that a finitely generated group is hyperbolic if and only if any of its Cayley graphs with respect to a finite generating set are hyperbolic metric spaces.
As is usual, for a path \(p\) in a metric space we will denote by \(p_-\) and \(p_+\) the initial and terminal endpoints of \(p\).
For the sake of simplicity, we will assume that all paths are continuous in this paper.

\begin{definition}[Quasigeodesic]
    Let \(\lambda \geq 1\) and \(c \geq 0\).
    A path \(p \colon [a,b] \to X\) in geodesic metric space \(X\) is called \emph{\((\lambda,c)\)-quasigeodesic} if for any subpath \(q\) of \(p\),
    \[
        \ell(q) \leq \lambda \dist(q_-,q_+) + c,
    \]
    where \(\ell(q)\) denotes the length of the subpath \(q\).
\end{definition}

\begin{definition}[Broken line]
    Let \(X\) be a geodesic metric space.
    A \emph{broken line} in \(X\) is a path \(p\) that has a (fixed) decomposition as a sequence of geodesic subpaths \(p_1, \dots, p_n\), with \((p_i)_+ = (p_{i+1})_-\) for each \(i = 1, \dots, n-1\).
\end{definition}

For three points \(x,y,z\) in metric space \(X\), the \emph{Gromov product} of \(x\) and \(y\) based at \(z\) is a measure of the angle between \(x\) and \(y\) measured from \(z\). It is given by the formula
\[
    \langle x, y \rangle_z = \frac{1}{2} \Big( \dist(x,z) + \dist(y,z) - \dist(x,y) \Big) .
\]
The following gives a useful criterion for broken lines to be quasigeodesics with well-controlled constants in hyperbolic metric spaces (cf. \cite[Lemma 4.2]{Ashot_res_hom_fin}).

\begin{lemma}[{\cite[Lemma 4.10]{MinMin}}]
\label{lem:broken_line_qgd}
    Suppose that \(X\) is a \(\delta\)-hyperbolic metric space.
    Let \(c_0, c_1,\) and \(c_2\) be constants such that \(c_0 \geq 14\delta\), $c_1 = 12( c_0+\delta)+1$ and $c_2=10 (\delta+c_1)$.
    Suppose that \(p=p_1 \dots p_n\) is a broken line in \(X\), where \(p_i\) is a geodesic with \((p_i)_-=x_{i-1} \), \((p_i)_+=x_i \), $i=1,\dots,n$.
    If \(\dist(x_{i-1}, x_{i}) \geq c_1 \) for \(i = 1, \dots, n\), and \(\langle x_{i-1}, x_{i+1} \rangle_{x_i} \leq c_0\) for each \(i = 1, \dots, n-1\), then the path \(p\) is \((4, c_2)\)-quasigeodesic.
\end{lemma}

\begin{definition}(Quasiconvex subgroup)
    A subspace \(Y\) of a geodesic metric space \(X\) is called \emph{\(\sigma\)-quasiconvex} if there is some \(\sigma \geq 0\) such that every geodesic in \(X\) with endpoints in \(Y\) lies in a \(\sigma\)-neighbourhood of \(Y\).
    
    Let \(G\) be a hyperbolic group, and \(Q \leqslant G\) a finitely generated subgroup.
    We say that \(Q\) is \emph{quasiconvex} if there is \(\sigma \geq 0\) such that, as a subset of \(G\), it is \(\sigma\)-quasiconvex in a Cayley graph \(\Gamma\) of \(G\) with respect to a finite generating set.
    Such \(\sigma\) is called a \emph{quasiconvexity constant} of \(Q\) with respect to \(\Gamma\).
\end{definition}

Note that while being quasiconvex does not depend on the choice of Cayley graph for subgroups of hyperbolic groups, the particular quasiconvexity constant does.
The following lemma tells us that quasiconvex subgroups are in a sense `orthogonal', modulo their intersections.
Given \(G\) and a Cayley graph \(\Gamma\) for \(G\) equipped with the edge-path metric \(\dist\), we write
\[
    \abs{g} = \dist(1,g)
\]
for an element \(g \in G\).

\begin{lemma}[{\cite[Lemma 2.3]{MinGFERF}}]
\label{lem:bdd_gromov_product}
    Let \(G\) be a hyperbolic group and let \(H, K \leqslant G\) be quasiconvex subgroups.
    Suppose that \(\delta \geq 0\) is a hyperbolicity constant for a Cayley graph \(\Gamma\) of \(G\) and \(\sigma \geq 0\) is a quasiconvexity constant for \(H\) and \(K\) with respect to \(\Gamma\).
    There is a constant \(c_0 = c_0(\delta,\sigma) \geq 0\) such that if \(h \in H\) is such that \(\abs{h}\) is minimal among elements of \(h(H \cap K)\), then \(\langle 1, hk \rangle_h \leq c_0\) for any \(k \in K\).
\end{lemma}

The next two lemmas are elementary statements dealing with double cosets.

\begin{lemma}
\label{lem:double_coset_int_central_is_subgp}
    Let \(G\) be a group with central subgroup \(Z \leqslant Z(G)\).
    Let \(H, K \leqslant G\) be finitely generated subgroups.
    Then \(HK \cap Z\) is a subgroup of \(Z\). 
\end{lemma}

\begin{proof}
    Let \(h_1, h_2 \in H, k_1, k_2 \in K\) be such that \(h_i k_i = z_i\) is an element of \(Z\) for \(i = 1, 2\).
    Since \(Z\) is central, \(h_i = z_i k_i^{-1} = k_i^{-1} z_i\).
    It follows that \(h_i k_i = z_i = k_i h_i\), so \(h_i\) and \(k_i\) commute and their product is central in \(G\).
    In particular, \(k_2^{-1} h_2^{-1} = h_2^{-1} k_2^{-1}\) is central.
    Thus \(h_1 k_1 k_2^{-1} h_2^{-1} = h_1 h_2^{-1} k_2^{-1} k_1 \in HK\), concluding the lemma.
\end{proof}

\begin{lemma}
\label{lem:int_of_projections}
    Let \(G\) be a group, \(H, K \leqslant G\) subgroups, and \(\pi \colon G \to Q\) a homomorphism.
    Write \(Z = \ker \pi\).
    If \(HK \cap Z = \{1\}\), then \(\pi(H) \cap \pi(K) = \pi(H \cap K)\).
    Moreover, if \(H \cap Z = \{1\}\) and \(K \cap Z = \{1\}\), the converse also holds.
\end{lemma}

\begin{proof}
    Suppose that \(g \in \pi(H) \cap \pi(K)\).
    There are elements \(h \in H\) and \(k \in K\) with \(\pi(h) = \pi(k) = g\).
    Thus \(h = zk\) for some \(z \in Z\).
    But then \(hk^{-1} = z \in HK \cap Z\), which is trivial by the hypothesis.
    Therefore, \(h = k \in H \cap K\) and so \(g \in \pi(H \cap K)\).
    The reverse inclusion is immediate.

    For the latter statement, suppose that \(H \cap Z = K \cap Z = \{1\}\) and let \(g \in HK \cap Z\).
    Then \(g = hk\) for some \(h \in H\) and \(k \in K\).
    As \(g \in Z\), we have \(\pi(h) = \pi(k)^{-1}\), which is an element of \(\pi(H) \cap \pi(K) = \pi(H \cap K)\).
    Therefore \(\pi(h) = \pi(x)\) for some \(x \in H \cap K\).
    Since \(H \cap Z = \{1\}\), we must have \(h = x \in H \cap K\).
    But then \(g \in K \cap Z = \{1\}\).
\end{proof}

\subsection*{Extensions and the profinite topology}

For the remainder of the section, assume that \(G\) is an extension of the group \(Q\) by the group \(Z\).
That is, \(G\) fits into the short exact sequence
\[
    1 \to Z \to G \to Q \to 1.
\]
The following will give us a criterion for separability of subgroups is extensions.

\begin{lemma}
\label{lem:sep_subgp_of_ext}
    Let \(H \leqslant G\) be a subgroup.
    If
    \begin{enumerate}
        \item[(i)] for any \(K \leqslant_f Z\), the double coset \(HK\) is separable in \(G\); and
        \item[(ii)] \(H \cap Z\) is separable in \(Z\),
    \end{enumerate}
    then \(H\) is separable in \(G\).
\end{lemma}

\begin{proof}
    Let \(g \notin H\) be arbitrary.
    We will show that there is some \(K \leqslant_f Z\) with \(g \notin HK\), from which it follows that \(H\) is separable by (i).
    Indeed, suppose otherwise, for a contradiction.
    Certainly, \(g \in HZ\), so \(g = hz\) for some \(h \in H\) and \(z \in Z\).
    If \(g \in HK\), then there is \(h' \in H\) and \(k \in K\) with \(g = h' k\).
    But then \(hz = h' k\), so \(h^{-1} h' = z k^{-1} \in H \cap Z\).
    It follows that \(z = h^{-1} h_n z_n \in (H \cap Z) K\).
    Therefore, if \(g \in HK\) for every \(K \leqslant_f Z\), we have \(z \in (H \cap Z)K\) for every such \(K\) also.
    As \(H \cap Z\) is separable in \(Z\) by (ii), this implies that \(z \in H \cap Z\).
    But then \(g = hz \in H\), which is our desired contradiction.
\end{proof}

We call \(Q\) \emph{highly residually finite} if \(G\) is residually finite whenever \(Z\) is finite.
This property was introduced by Corson and Ratkovich under the moniker `super residual finiteness' \cite{Corson-Ratkovich}.
Remarkably, highly residually finite groups are exactly the class of groups that are residually finite and have separable cohomology in degree 2 \cite{Lorenson}.

A group is called \emph{strongly subgroup separable} (or \emph{ERF} -- extended residually finite) if all of its subgroups are separable.
Allenby and Gregorac claimed that a split extension of a subgroup separable group by a finitely generated strongly subgroup separable group is again subgroup separable \cite{Allenby-Gregorac}.
In the same spirit we can show that certain non-split extensions are subgroup separable.
We also sharpen the result, allowing for the consideration of subgroups in a particular class.
Note that \cite{Allenby-Gregorac} does not contain a correct proof of the split extension case, so we include it here for convenience.

Let \(\mathfrak P\) be a group property, and say a group \(G\) is \emph{\(\mathfrak P\)-subgroup separable} if each of its finitely generated subgroups satisfying \(\mathfrak P\) are closed in the profinite topology on \(G\).
Similarly, call \(G\) \emph{strongly \(\mathfrak P\)-subgroup separable} if each of its subgroups satisfying \(\mathfrak P\) is separable.
Taking \(\mathfrak P\) to be the empty property recovers ordinary subgroup separability.

\begin{proposition}
\label{lem:ERF_by_LERF_good}
    Let \(\mathfrak P\) be a group property that is preserved under taking subgroups and quotients.
    Suppose \(Q\) is \(\mathfrak P\)-subgroup separable and \(Z\) is finitely generated and strongly \(\mathfrak P\)-subgroup separable.
    If either 
    \begin{enumerate}
        \item[(a)] \(Q\) is highly residually finite; or
        \item[(b)] \(G = Z \rtimes Q\) is a split extension,
    \end{enumerate}
    then \(G\) is \(\mathfrak P\)-subgroup separable.
\end{proposition}

\begin{proof}
    Let \(H \leqslant G\) be a finitely generated subgroup of \(G\) satisfying \(\mathfrak P\).
    We will verify the hypotheses of Lemma~\ref{lem:sep_subgp_of_ext} for \(H\).
    To verify (i), let \(K \leqslant_f Z\) be a finite index subgroup of \(Z\).
    Since \(Z\) is finitely generated, there is a finite index characteristic subgroup \(Z' \leqslant_f Z\) with \(Z' \subseteq K\).
    It will suffice to show that \(HZ'\) is separable in \(G\), for \(HK\) is a finite union of right cosets of \(HZ'\).

    As \(Z'\) is characteristic in \(Z \lhd G\), \(Z'\) is normal in \(G\).
    Thus we have the exact sequence 
    \[1 \to Z/Z' \to G/Z' \to Q \to 1.\]
    Note that \(Z/Z'\) is a finite group.
    If \(Q\) is highly residually finite, then \(G/Z'\) is residually finite by definition.
    Otherwise, \(G/Z'\) is a finite split extension of a residually finite group, and is therefore also residually finite.

    In either case, there is a finite index subgroup \(G' \leqslant_f G/Z'\) with \(G' \cap Z/Z' = \{1\}\).
    That is, \(G'\) may be identified with a subgroup of finite index in \(Q\).
    Now \(G/Z'\) and \(Q\) are commensurate, so the assumption that \(Q\) is \(\mathfrak P\)-subgroup separable tells us that \(G/Z'\) is \(\mathfrak P\)-subgroup separable.
    The image of \(H\) in \(G/Z'\) is thus separable, as it is a finitely generated subgroup satisfying \(\mathfrak P\).
    As homomorphisms are open maps onto their images with respect to the profinite topology, the full preimage \(HZ'\) is closed in the profinite topology on \(G\), verifying (i).
    Finally, condition (ii) follows from the assumption that \(Z\) is strongly \(\mathfrak P\)-subgroup separable and that \(\mathfrak P\) is subgroup closed.
\end{proof}

Common examples of properties as in the statement above include being cyclic, abelian, or solvable.

\begin{remark}
\label{rem:central_ext_fp_quotient}
    For our main results, we will need the kernels of central extensions to be finitely generated.
    Finite generation of \(G\) together with finite presentability of the quotient \(Q\) implies this condition. 
    Indeed, in this case \(Z\) is finitely normally generated in \(G\) which, given that \(Z\) is central in \(G\), means that \(Z\) is finitely generated.
    Note that every countable abelian group is the centre of some finitely generated group \cite{Ould}.
\end{remark}

We conclude with a series of elementary observations about the separability of products of subgroups in groups.

\begin{lemma}
\label{lem:prod_with_subgroup_rz_quotient}
    Let \(H_1, \dots, H_n \leqslant G\) be finitely generated subgroups.
    If \(K \lhd G\) is normal and \(G/K\) is \(n\)-coset separable, then \(H_1 \dots H_n K\) is separable in \(G\).
\end{lemma}

\begin{proof}
    Let \(\rho \colon G \to G/K\) be the quotient map.
    The subset \(H_1 \dots H_n K\) is the full preimage of \(\rho(H_1 \dots H_n) = \rho(H_1) \dots \rho(H_n)\), which is separable in \(G/K\) by the hypothesis.
    As homomorphisms are continuous with respect to the profinite topology, \(H_1 \dots H_n K\) is separable in \(G\).
\end{proof}

\begin{lemma}
\label{lem:prod_sep_if_int_of_subgp_nontriv}
    Let \(H_1, \dots, H_n \leqslant G\) be finitely generated subgroups.
    Suppose that the product \(H_1 \dots H_{n}\) contains a normal subgroup \(K \lhd G\) such that \(G/K\) is \(n\)-coset separable.
    Then the product \(H_1 \dots H_n\) is separable.
\end{lemma}

\begin{proof}
    It follows from the hypotheses that \[H_1 \dots H_n = H_1 \dots H_n K. \]
    The result is now a consequence of Lemma~\ref{lem:prod_with_subgroup_rz_quotient}.
\end{proof}

\begin{lemma}
\label{lem:prod_with_Z_separable}
    Suppose that \(G\) is subgroup separable, \(Z\) is finitely generated, and \(Q\) is \(n\)-coset separable.
    Let \(H_1, \dots, H_n \leqslant G\) be finitely generated subgroups.
    If \(K \leqslant_f Z\), then the product \(H_1 \dots H_n K\) is separable in \(G\).
\end{lemma}

\begin{proof}
    Let \(K \leqslant_f Z\) be an arbitrary finite index subgroup of \(Z\).
    Since \(Z\) is finitely generated, there is \(K' \leqslant_f K\) which is characteristic in \(Z\), and hence normal in \(G\).
    Since \(G\) is subgroup separable, \(K'\) is separable in \(G\).
    Therefore there is \(G' \leqslant_f G\) with \(G' \cap Z = K'\).
    Since \(K'\) is normal in \(G\), we may take \(G'\) to be normal in \(G\) also (replacing it by its normal core if necessary).
    Thus \(G'\) fits into the exact sequence \(1 \to K' \to G' \to Q' \to 1\), where \(Q' = \pi(G') \leqslant_f Q\) is a finite index subgroup of \(Q\).
    As \(Q'\) is a subgroup of \(Q\), it is also \(n\)-coset separable.
    
    For each \(i\), let \(T_i = G' \cap H_i \leqslant_f H_i\).
    Since \(G'\) is normal, any \(G\)-conjugate of \(T_i\) is contained in \(G'\).
    Hence the product \(T_1^{g_1} \dots T_n^{g_n}\) is contained in \(G'\) for any \(g_1, \dots, g_n \in G\).
    Now \(G'/K' \cong Q'\) is \(n\)-coset separable so by Lemma~\ref{lem:prod_with_subgroup_rz_quotient}, the product \(T_1^{g_1} \dots T_n^{g_n} K'\) is separable in \(G'\).
    As \(H_1 \dots H_n K\) is a finite union of translates of such products, the lemma follows from this statement.
\end{proof}

\section{Double coset separability}

In this section we will prove Theorems~\ref{thm:central_dblcost} and \ref{thm:prod_with_nilp}.
Let \(G\) be an extension of \(Q\) by \(Z\), so that \(\pi \colon G \to Q\) is a surjective homomorphism with kernel \(Z = \ker \pi \leqslant G\).

\begin{lemma}
\label{lem:subgp_Z_double_coset_sep}
    Suppose \(G\) and \(Q\) are subgroup separable, and that \(Z\) is finitely generated.
    Let \(H \leqslant G\) be finitely generated and let \(K \leqslant Z\) be an arbitrary subgroup.
    If \((H \cap Z)K\) is separable in \(Z\), then \(HK\) is separable in \(G\).
\end{lemma}

\begin{proof}
    Let \(Z_n \lhd_f Z\) be an enumeration of normal cores of the finite index subgroups of \(Z\). 
    By our assumption, \((H \cap Z)K\) is separable in \(Z\), so that
    \begin{equation}
    \label{eq:approx_dbl_coset}
        \bigcap_{n \geq 1} (H \cap Z) K Z_n = (H \cap Z) K.
    \end{equation}
    
    For each \(n\), Lemma~\ref{lem:prod_with_Z_separable} tells us that the double coset \(H Z_n\) is separable in \(G\).
    Since \(Z_n \lhd_f Z\) and \(K \leqslant Z\), the triple coset \(HKZ_n\) is a finite union of right translates of \(H Z_n\), and is thus separable in \(G\) also.
    We will show that the double coset \(HK\) is approximated by such triple cosets.
    
    Suppose that \(g \in HKZ_n\) for every \(n \geq 1\), so that \(g = h_n k_n z_n\) with \(h_n \in H, k_n \in K, z_n \in Z_n\).
    We write \(h = h_1, k = k_1,\) and \(z = z_1\) for ease of reading.
    Then for each \(n \geq 1\),
    \[
        h_n^{-1} h = k_n z_n z^{-1} k^{-1} \in H \cap Z.
    \]
    Rearranging, then, we have for all \(n \geq 1\)
    \[
        kzk^{-1} = (h^{-1} h_n) (k_n k^{-1}) (k z_n k^{-1}) \in (H \cap Z)KZ_n,
    \]
    where the inclusion uses the fact that \(Z_n \lhd Z\), so \(kz_nk^{-1} \in Z_n\).
    It follows from (\ref{eq:approx_dbl_coset}) that \(kzk^{-1} \in (H \cap Z)K\).
    But then \(g = hkz = hkzk^{-1}k \in H(H \cap Z)K K = HK\).
    This shows that \(HK = \bigcap_{n \geq 1} HKZ_n\) is closed.
\end{proof}

As an immediate consequence, we obtain the following.

\begin{proposition}
\label{prop:quotient_ext_LERF}
    Suppose \(G\) and \(Q\) are subgroup separable, \(Z\) is slender and double coset separable, and \(K \leqslant Z\).
    If \(K \lhd G\) then \(G/K\) is subgroup separable.
\end{proposition}

\begin{proof}
    Under the quotient map \(G \to G/K\), the full preimage of any finitely generated subgroup of \(G/K\) is a double coset \(HK\), where \(H\) is a finitely generated subgroup of \(G\).
    By Lemma~\ref{lem:subgp_Z_double_coset_sep}, \(HK\) is separable in \(G\).
    Now, homomorphisms are open maps onto their images with respect to the profinite topologies, yielding the result.
\end{proof}

In the special case of a central extension, the kernel of the extension is abelian and so product separable.
Thus the hypotheses of the above proposition is met, given that the group and its quotient are subgroup separable.

\begin{corollary}
\label{lem:central_quotient_LERF}
    Let \(G\) be a central extension of a subgroup separable group by finitely generated \(Z \leqslant Z(G)\).
    If \(G\) is subgroup separable then so is \(G/K\) for any \(K \leqslant Z\).
\end{corollary}

The following is implicitly proven as \cite[Theorem 3.3]{You}; the result is only stated for \(Q\) a free group but the proof only uses the fact that it is double coset separable.

\begin{proposition}
\label{prop:double_coset_sep_times_Z}
    Suppose \(Q\) is double coset separable and \(Z = \langle a \rangle\) is infinite cyclic.
    Then the direct product \(Q \times Z\) is double coset separable.
\end{proposition}

It is an immediate consequence that \(Q \times A\) is double coset separable whenever \(Q\) is double coset separable and \(A\) is finitely generated abelian.
As stated in the introduction, we are able to extend this result in a couple of ways.

\begin{lemma}
\label{lem:dbl_coset_sep_if_int_trivial}
    Suppose \(Z\) is finitely generated and residually finite.
    Let \(H, K \leqslant G\) be subgroups with \(HK \cap Z = \{1\}\).
    If \(G/Z'\) is double coset separable for every \(Z' \leqslant_f Z\) with \(Z' \lhd G\), then \(HK\) is separable in \(G\).
\end{lemma}

\begin{proof}
    As \(Z\) is finitely generated and residually finite, there is a sequence \(\{Z_n\, | \, n \in \NN\}\) of finite index characteristic subgroups of \(Z\) with \(\bigcap_{n\geq1}Z_n = \{1\}\).
    Moreover, as each \(Z_n\) characteristic in \(Z\), we have that \(Z_n \lhd G\).
    By assumption then, the groups \(G/Z_n\) are double coset separable.
    Hence by Lemma~\ref{lem:prod_with_subgroup_rz_quotient}, each of the triple cosets \(H K Z_n\) is separable in \(G\).

    Suppose that \(g \notin HK\) but \(g\) belongs to the profinite closure of \(HK\).
    As \(G/Z\) is double coset separable by assumption, the triple coset \(HKZ\) is separable in \(G\).
    Since this subset contains \(HK\), we must have \(g \in HKZ\).
    That is, we have \(g = h k z\) for some \(h \in H, k \in K,\) and \(z \in Z\).
    Similarly, \(g \in H K Z_n = H Z_n K\) for every \(n \geq 1\).
    That is, for each \(n \geq 1\), we have \(g = h_n z_n k_n\), where \(h_n \in H, k_n \in K,\) and \(z_n \in Z_n\).
    It follows that
    \[
        h_n^{-1} h k k_n^{-1} = z_n (k_n z^{-1} k_n^{-1}) \in H K \cap Z,
    \]
    where the membership is due to the fact that \(Z\) is normal in \(G\).
    By assumption \(H K \cap Z = \{1\}\), so that \(z_n k_n z^{-1} k_n^{-1} = 1\) for every \(n \geq 1\).
    But then \(z = k_n^{-1} z_n k_n \in Z_n\) for all \(n \geq 1\).
    As \(\bigcap_{n \geq 1} Z_n = \{1\}\), we conclude that \(z = 1\) and so \(g \in H K\) as required.
\end{proof}

We are now ready to prove Theorem~\ref{thm:central_dblcost}.

\begin{proof}[Proof of Theorem~\ref{thm:central_dblcost}]
    Suppose that \(G\) is a central extension of a double coset separable group \(Q\) by a finitely generated group \(Z\).
    We will induct on the free rank \(r(Z)\) of \(Z\).
    If \(r(Z) = 0\), then \(Z\) is finite.
    In this case, \(G\) is residually finite by assumption and is thus commensurate to \(Q\), which is double coset separable.
    Suppose then that \(r(Z) > 0\) and that the statement holds for all lesser values.
    The group \(Z\) has a finite index free abelian subgroup \(Z'\).
    As \(G\) is subgroup separable, there is a finite index subgroup \(G' \leqslant_f G\) such that \(G' \cap Z = Z'\).
    We write \(Q' = \pi(G') \leqslant_f Q\), which is also double coset separable.
    The sequence \(1 \to Z' \to G' \to Q' \to 1\) is thus again a short exact sequence for a central extension, and \(G\) is double coset separable if and only if \(G'\) is.
    We may therefore assume without loss of generality that \(Z = Z'\) is free abelian (so \(G = G'\) and \(Q = Q'\)).

    Let \(H_1, H_2 \leqslant G\) be finitely generated subgroups of \(G\).
    If the intersection \(K = H_1H_2 \cap Z\) is nontrivial, then it is a subgroup of \(Z\) by Lemma~\ref{lem:double_coset_int_central_is_subgp}.
    Moreover, as \(Z\) is torsionfree, the subgroup \(K\) is infinite in this case.
    
    We have the formula of ranks \(r(Z) = r(K) + r(Z/K)\).
    As \(K\) is infinite, its rank is at least \(1\) and therefore \(r(Z/K) < r(Z)\).
    Now the group \(G/K\) fits into the short exact sequence \(1 \to Z/K \to G/K \to Q \to 1\) and is subgroup separable by Corollary~\ref{lem:central_quotient_LERF}.
    The induction hypothesis thus implies that \(G/K\) is double coset separable.
    But then \(H_1H_2\) is separable by Lemma~\ref{lem:prod_sep_if_int_of_subgp_nontriv}, so we are done.
    
    Suppose instead that \(H_1H_2 \cap Z = \{1\}\).
    We need only verify the hypotheses of Lemma~\ref{lem:dbl_coset_sep_if_int_trivial}.
    Let \(Z' \leqslant_f Z\) be such that \(Z' \lhd G\).
    As \(Z'\) has finite index in \(Z\), we have \(r(Z/Z') = 0\).
    In particular, \(G/Z'\) is a double coset separable by the base case of the induction.
    Therefore Lemma~\ref{lem:dbl_coset_sep_if_int_trivial} shows that \(H_1 H_2\) is separable in \(G\).
\end{proof}

When the extension above is split (i.e. \(G\) is a direct product of \(Q\) and \(Z\)), \(G\) must be subgroup separable, so we recover the result of You.
In fact, we may extend the result from abelian to nilpotent groups as in Theorem~\ref{thm:prod_with_nilp}.
The proof of this proceeds essentially verbatim to that above; one need only replace the word `abelian' with `polycyclic', the inductive quantity of free abelian rank with that of the Hirsch length, and the role of \(Z\) is taken by the centre of the nilpotent factor \(N\) (which is necessarily infinite).

Note that this proof only uses the a priori weaker hypotheses that \(N\) is polycyclic and every infinite quotient of \(N\) has a finite index torsionfree subgroup with infinite centre.
However, it is straightforward to see that this class coincides with the class of finitely generated virtually nilpotent groups.

\section{Product separability}

Our main tool for showing product separability for extensions will be a combinatorial property, which roughly states that essentially any way one writes a given element as a product of elements from some fixed collection of subgroups has one of the factors lying in a finite set.

Observe that, given subgroups \(H, K\) and \(g \in HK\), there may be many ways to write \(g\) as a product of an element of \(H\) and one of \(K\).
Indeed, if \(g = hk\) is such a product and \(H \cap K\) is infinite, \((hs)(s^{-1}k)\) gives infinitely many different decompositions of \(g\) as \(s\) ranges over \(H \cap K\).
This motivates the following relation on such products.

\begin{definition}[Product representatives]
    Let \(G\) be a group and \(H_1, \dots, H_n \leqslant G\) be subgroups.
    Given \(g \in H_1 \dots H_n\), a \emph{product representative} of \(g\) is a tuple \((h_1, \dots, h_n) \in H_1 \times \dots \times H_n\) such that \(g = h_1 \dots h_n\).

    For \(i = 1, \dots, n-1\), write \(S_i = H_i \cap H_{i+1}\).
    We say that two product representatives \((h_1, \dots, h_n)\) and \((k_1, \dots, k_n)\) are equivalent if there are elements \(s_1 \in S_1, \dots, s_{n-1} \in S_{n-1}\) such that \(s_{i-1} k_i = h_i s_i\) for each \(i = 1, \dots, n\) (where \(s_0 = s_n = 1\)).
\end{definition}

The relation defined above is easily seen to be an equivalence relation.

\begin{definition}[Bottlenecked product representatives]
    Let \(G\) be a group and \(\mathcal{C}\) be a class of subgroups.
    We say that \(G\) has \emph{bottlenecked product representatives} over \(\mathcal{C}\) if for any \(H_1, \dots, H_n \in \mathcal{C}\), any \(g \in H_1 \dots H_n\), there is a set \(R \subseteq H_1 \times \dots \times H_n\) of product representatives of \(g\) such that the projection \(R \to H_i\) is finite for some \(i = 1, \dots, n\), and every product representative of \(g\) is equivalent to some product representative in \(R\).
\end{definition}

\begin{remark}
    We note that the property of having bottlenecked product representatives even over cyclic subgroups appears to be quite restrictive.
    Indeed, suppose that \(G\) contains a non-cyclic free abelian subgroup.
    Let \(H = \langle a \rangle\) and \(K = \langle b \rangle\) be distinct summands of this subgroup.
    Then \(1 = a^n b^n a^{-n} b^{-n}\) form infinitely many product representatives of the identity element in the quadruple coset \(HKHK\), and since \(H \cap K\) is trivial, each lies in its own one-element equivalence class.
    Hence the identity is a witness to the failure of this property for such a group \(G\).
\end{remark}

We are now ready to prove our main technical theorem.

\begin{theorem}
\label{thm:main}
    Let \(G\) be a central extension of a product separable group \(Q\) by a finitely generated group.
    If \(G\) is subgroup separable and \(Q\) has bottlenecked product representatives over finitely generated subgroups, then \(G\) is product separable.
\end{theorem}

\begin{proof}
    Let \(1 \to Z \to G \to Q \to 1\) be the short exact sequence corresponding to the extension, so \(Z \leqslant Z(G)\), and denote the quotient map by \(\pi \colon G \to Q\).
    We will begin with a basic reduction.
    As \(Z\) is finitely generated abelian, it has a free abelian subgroup \(Z'\) of finite index.
    As \(G\) is subgroup separable, it has a finite index subgroup \(G'\) for which \(G' \cap Z = Z'\).
    Now \(G\) is product separable if and only if \(G'\) is, so we may assume without loss of generality that \(Z\) is itself free abelian.
    
    We proceed by a double induction on the number of subgroups \(n\) and the rank \(r\) of \(Z\) as a free abelian group.
    The base cases \(n = 0\) or \(r = 0\) hold trivially: respectively these are the statements that \(G\) is residually finite, and that \(G\) is isomorphic to \(Q\), which is product separable by assumption.

    Suppose \(n \geq 1, r \geq 1\) and that the statement holds for \(n' < n\) or \(r' < r\). 
    Let \(H_1, \dots, H_n \leqslant G\) be finitely generated subgroups of \(G\).

    Lemma~\ref{lem:double_coset_int_central_is_subgp} tells us that for each \(i = 1, \dots, n-1\), the intersection \(K = H_i H_{i+1} \cap Z\) is a subgroup of \(Z\).
    Suppose that \(K\) is nontrivial.
    As \(Z\) is central, \(K\) is normal in \(G\) and, as \(Z\) is free abelian, \(K\) is infinite.
    Therefore \(Z/K\) is an abelian group of rank strictly less than \(r\) and we have the exact sequence
    \[
        1 \to Z/K \to G/K \to Q \to 1
    \]
    which is a central extension of \(Q\) by \(Z/K\).
    Moreover, \(G/K\) is subgroup separable by Corollary~\ref{lem:central_quotient_LERF}. 
    Thus, by the induction hypothesis, \(G/K\) is product separable.
    It follows from Lemma~\ref{lem:prod_sep_if_int_of_subgp_nontriv} that \(H_1 \dots H_n\) is separable.
    We may thus assume that \(H_i H_{i+1} \cap Z = \{1\}\) for each \(i = 1, \dots, n-1\).
    It follows that \(H_i \cap Z = \{1\}\) for each \(i = 1, \dots, n\), and so each of the projections \(H_i \to \pi(H_i)\) are injective.
    Moreover, by Lemma~\ref{lem:int_of_projections} we have
    \begin{equation}
    \label{eq:int_commutes}
        \pi(H_i) \cap \pi(H_{i+1}) = \pi(H_i \cap H_{i+1})
    \end{equation}
    for each \(i = 1, \dots, n-1\).

    We will approximate \(H_1 \dots H_n\) by products of the form \(H_1 \dots H_n Z_m\), where \(Z_m\) denotes the finite index subgroup \(\{z^m \, | \, z \in Z\} \leqslant_f Z\).
    By Lemma~\ref{lem:prod_with_Z_separable}, all such products are separable.
    Suppose for a contradiction that \(g \notin H_1 \dots H_n\) is in the closure of \(H_1 \dots H_n\) with respect to the profinite topology. 
    It follows that \[g \in \bigcap_{m \geq 1} H_1 \dots H_n Z_m.\]
    By definition, we have for each \(m \geq 1\)
    \begin{equation*}
        g = h_1^{(m)} \dots h_n^{(m)} z^{{(m)}}
    \end{equation*}
    where \(h_i^{(m)} \in H_i\) and \(z^{(m)} \in Z_m\).

    For each \(i\) and \(m\), let us write \(x_i^{(m)} = \pi(h_i^{(m)})\), so that \((x_1^{(m)}, \dots, x_n^{(m)})\) is a product representative for \(\pi(g)\) for each \(m \geq 1\).
    Now \(Q\) has bottlenecked products representatives over finitely generated subgroups; let \(R\) be the set of equivalence class representatives of product representatives of \(G\) given by the definition, with the projection \(R \to \pi(H_j)\) finite for some \(j = 1, \dots, n\).
    That is, there is some product representative \((y_1^{(m)}, \dots, y_n^{(m)})\) of \(\pi(g)\) equivalent to \((x_1^{(m)}, \dots, x_n^{(m)})\) such that 
    \begin{equation}
    \label{eq:prod_reps_finite}
        \abs{\Big\{y_j^{(m)} \, | \, m \geq 1\Big\}} < \infty.
    \end{equation}
    
    By definition, there are elements \(t^{(m)}_i \in \pi(H_i) \cap \pi(H_{i+1})\) with \(t^{(m)}_{i-1} y^{(m)}_i = x^{(m)}_i t^{(m)}_i\) for \(i = 1, \dots, n\) (writing \(t^{(m)}_0 = t^{(m)}_n = 1\)).
    As \(y_i^{(m)} \in \pi(H_i)\), there is \(k_i^{(m)} \in H_i\) with \(\pi(k_i^{(m)}) = y_i^{(m)}\).
    Moreover, by (\ref{eq:int_commutes}), \(t_i^{(m)} \in \pi(H_i \cap H_{i+1})\), so there is some \(s_i^{(m)} \in H_i \cap H_{i+1}\) with \(\pi(s_i^{(m)}) = t_i^{(m)}\) (again taking \(s_0^{(m)} = s_n^{(m)} = 1\)).
    Now for \(i = 1, \dots, n\), we have
    \begin{align*}
        \pi(s_{i-1}^{(m)} k_i^{(m)}) &= t^{(m)}_{i-1} y_i^{(m)} \\
            &= x^{(m)}_i t^{(m)}_i \\
            &= \pi(h_i^{(m)} s_i^{(m)}).
    \end{align*}
    However, we have observed that \(\pi\) is injective restricted to \(H_i\), so in fact \(s_{i-1}^{(m)} k_i^{(m)} = h_i^{(m)} s_i^{(m)}\) for each \(i\).
    This implies that \((h_1^{(m)}, \dots, h_n^{(m)})\) and \((k_1^{(m)}, \dots, k_n^{(m)})\) are equivalent product representatives in \(H_1 \dots H_n\).
    That is, we have
    \[
        g = k_1^{(m)} \dots k_n^{(m)} z^{(m)}
    \]
    for each \(m \geq 1\).

    Now by (\ref{eq:prod_reps_finite}), we may pass to a subsequence for which \(y_j = y^{(m)}_j\) is constant.
    The fact that the projection \(H_j \to \pi(H_j)\) is injective implies that the elements \(k_j = k_j^{(m)}\) are constant also.
    Thus \(g \in H_1 \dots H_{j-1} k_j H_{j+1} \dots H_n Z_m\) for each \(m \geq 1\).
    By the induction hypothesis the set \(H_1 \dots H_{j-1} k_j H_{j+1} \dots H_n\) is separable, so there is some \(N \lhd_f G\) with \(g \notin K = H_1 \dots H_{j-1} k_j H_{j+1} \dots H_n N\).
    Since \(N \cap Z \leqslant_f Z\), the subgroup \(N\) contains \(Z_m\) for large enough \(m \geq 1\).
    This implies \(K\) contains \(H_1 \dots H_{j-1} k_j H_{j+1} \dots H_n Z_m\) and thus \(g\), a contradiction.
\end{proof}

\begin{proposition}
\label{prop:hyp_bounded_prods}
    Let \(Q\) be a hyperbolic group.
    Then \(Q\) has bottlenecked product representatives over the class of quasiconvex subgroups.
\end{proposition}

\begin{proof}
    Let \(\delta\) a hyperbolicity constant for a Cayley graph \(\Gamma\) of \(Q\) with respect to a finite generating set.
    Let \(H_1, \dots, H_n \leqslant Q\) be quasiconvex subgroups with quasiconvexity constant \(\sigma \geq 0\), and let \(g \in H_1 \dots H_n\).
    Take \(c_0 = c_0(\delta,\sigma)\) to be the constant of Lemma~\ref{lem:bdd_gromov_product}, and \(c_1 = c_1(c_0), c_2 = c_2(c_0)\) be the corresponding constants of Lemma~\ref{lem:broken_line_qgd}.
    Let \((h_1, \dots, h_n)\) be a product representative of \(g\).
    We will show that \((h_1, \dots, h_n)\) is equivalent to some \((k_1, \dots, k_n)\) with \(\abs{k_i}\) bounded for some \(i = 1, \dots, n\). 
    As the word metric is proper, this shows that \(Q\) has bottlenecked product representatives.

    There is \(k_1 \in h_1 (H_1 \cap H_2)\) such that \(\abs{k_1}\) is minimal among elements in this coset.
    By definition, \(k_1 = h_1 s_1\) for some \(s_1 \in H_1 \cap H_2\).
    We continue by induction: for \(i = 2, \dots, n-1\) take \(k_i \in s_{i-1} h_i (H_i \cap H_{i+1})\) to be such that \(\abs{k_i}\) is minimal in this coset.
    As before, there is \(s_i \in H_i \cap H_{i+1}\) such that \(s_{i-1} h_i = k_i s_i\).
    Finally let \(k_n = s_{n-1} h_n\).
    From the construction it is immediate that \((h_1, \dots, h_n)\) and \((k_1, \dots, k_n)\) are equivalent product representatives of \(g\).

    Now define the path \(p\) in \(\Gamma\) to be a broken line with nodes \(1, k_1, k_1 k_2, \dots, k_1 \dots k_n = g\).
    That is, \(p\) is a broken line whose initial vertex is the identity and whose segments are labelled by \(k_1, \dots, k_n\).
    By Lemma~\ref{lem:bdd_gromov_product}, we have that \(\langle 1, k_i k_{i+1} \rangle_{k_i} \leq c_0\) for each \(i = 1, \dots, n-1\).
    If there is some \(i = 1, \dots, n\) such that \(\abs{k_i} \leq c_1\) then we are done, so suppose otherwise.
    In this case, we may apply Lemma~\ref{lem:broken_line_qgd} to see that \(p\) is \((4,c_2)\)-quasigeodesic in \(\Gamma\).
    This gives, for each \(i = 1, \dots, n\),
    \[
        \abs{k_i} \leq \ell(p) \leq 4 \abs{p} + c_2 = 4 \abs{g} + c_2, 
    \]
    completing the proof.
\end{proof}

We can now combine the above to prove our main theorem.

\begin{proof}[Proof of Theorem~\ref{thm:main_intro}]
    Let \(G\) be a finitely generated central extension of subgroup separable locally quasiconvex hyperbolic group \(Q\) by a group \(Z\).
    It is well known that hyperbolic groups are finitely presented, so Remark~\ref{rem:central_ext_fp_quotient} tells us that \(Z\) is finitely generated.
    Proposition~\ref{prop:hyp_bounded_prods} tells us that \(Q\) has bottlenecked product representatives over the class of quasiconvex subgroups.
    Moreover, by the main result of \cite{MinGFERF}, the product of finitely many quasiconvex subgroups of \(Q\) is separable in \(Q\).
    As \(Q\) is locally quasiconvex, the above implies it has bottlenecked product representatives over all finitely generated subgroups and that it is product separable.
    Assuming \(G\) is subgroup separable, the result now follows immediately by an application of Theorem~\ref{thm:main}.
\end{proof}

Marginally more work gives us the corollary from the introduction.

\begin{proof}[Proof of Corollary~\ref{cor:ext_limitgp}]
    Let \(Q\) be a hyperbolic limit group and \(G\) a central extension of \(Q\) by a finitely generated group.
    All limit groups are subgroup separable by the work of Wilton \cite{WiltonLimitGps}.
    Moreover, limit groups are highly residually finite \cite{Grunewald-Jaikin-Zalesskii}, and so by Proposition~\ref{lem:ERF_by_LERF_good}, the extension \(G\) is subgroup separable.
    Finally, limit groups are locally quasiconvex \cite{Dahmani_Conv}.
    Hence, Theorem~\ref{thm:main_intro} applies to hyperbolic limit groups.

    Let \(M\) be a Seifert-fibred 3-manifold with hyperbolic orbifold base \(\Sigma\).
    The orbifold \(\Sigma\) has a finite-sheeted cover by a hyperbolic surface, which induces a finite-sheeted covering of \(M\).
    As the fundamental group of \(M\) is product separable if and only if the fundamental group of this covering is, we may suppose \(\Sigma\) is a hyperbolic surface.
    If the fibre group \(\pi_1 S^1\) has finite image in \(\pi_1 M\), then \(\pi_1 M\) and \(\pi_1 \Sigma\) are commensurable and we are done, so suppose otherwise.
    Then \(M\) has a two-sheeted cover (orientable) \(M'\) for which the image of \(\pi_1 S^1\) is central in \(\pi_1 M'\).
    Now \(\pi_1 M'\) is a central extension of \(\pi_1 \Sigma'\), for a two-sheeted orientable cover \(\Sigma'\) of \(\Sigma\).
    The result now follows as orientable surface groups are limit groups.
\end{proof}

% \section*{Statements and declarations}

% \subsection*{Data availability}
% Data sharing not applicable to this article as no datasets were generated or analysed during the current study.

% \subsection*{Conflict of interest}
% The author would like to state that there are no financial or non-financial interests that are directly or indirectly related to the work.

% \subsection*{Open Access}
% This article is licensed under a Creative Commons Attribution 4.0 International License, which
% permits use, sharing, adaptation, distribution and reproduction in any medium or format, as long as you give
% appropriate credit to the original author(s) and the source, provide a link to the Creative Commons licence,
% and indicate if changes were made. The images or other third party material in this article are included
% in the article’s Creative Commons licence, unless indicated otherwise in a credit line to the material. If
% material is not included in the article’s Creative Commons licence and your intended use is not permitted
% by statutory regulation or exceeds the permitted use, you will need to obtain permission directly from the
% copyright holder. To view a copy of this licence, visit \href{http://creativecommons.org/lice}{http://creativecommons.org/lice}.

\emergencystretch=1em

\printbibliography

\end{document}